\documentclass{article}
\usepackage[english]{babel}
\usepackage{amsrefs}
\usepackage{amssymb,amsmath,amsthm,amsfonts,epsfig,graphicx,psfrag,color}
\usepackage[utf8]{inputenc}
\usepackage{hyperref}
\hypersetup{
    colorlinks,
    citecolor=blue,
    filecolor=black,
    linkcolor=blue,
    urlcolor=magenta
}
\newcommand{\rojo}[1]{\textcolor{red}{#1}}

\DeclareMathOperator{\homeo}{H}
\DeclareMathOperator{\Diff}{Diff}

\DeclareMathOperator{\mesh}{mesh}

\DeclareMathOperator{\comp}{comp}

\theoremstyle{plain}
\newtheorem{thm}{Theorem}[section]
\newtheorem{cor}[thm]{Corollary}
\newtheorem{lem}[thm]{Lemma}
\newtheorem{prop}[thm]{Proposition}

\theoremstyle{definition}
\newtheorem{df}{Definition}[section]

\newtheorem{rmk}[thm]{Remark}

\theoremstyle{remark}

\DeclareMathOperator{\diam}{diam}

\DeclareMathOperator{\interior}{int}
\DeclareMathOperator{\dist}{dist}

\DeclareMathOperator{\cont}{C}

\newcommand{\R}{\mathbb R}

\newcommand{\C}{\mathcal C}
\newcommand{\Z}{\mathbb Z}
\newcommand{\N}{\mathbb N}

\newcommand{\T}{\mathbb{T}}
\newcommand{\Sph}{\mathbb{S}}
\renewcommand{\epsilon}{\varepsilon}

\newcommand{\expc}{\alpha}

\begin{document}

\author{Alfonso Artigue\footnote{Email: aartigue@litoralnorte.udelar.edu.uy. 
Adress: DMEL, Centro Universitario Regional Litoral Norte, 
Universidad de la República, Uruguay.}}
\title{Generic surface homeomorphisms are almost continuum-wise expansive}

\date{\today}
\maketitle

\begin{abstract}
We show that for a compact surface without boundary $M$ the set of cw-expansive homeomorphisms is dense in the set of all the homeomorphisms of $M$ with respect to the $C^0$ topology. After this we show that for a generic homeomorphism $f$ of $M$ it holds that: for all $\epsilon>0$ there is a cw-expansive homeomorphism $g$ of $M$ which is 
$\epsilon$-close to $f$ and is semiconjugate to $f$; moreover, if $\pi\colon M\to M$ is this semiconjugacy then $\pi^{-1}(x)$ is connected, does not separate $M$ and has diameter less than $\epsilon$ for all $x\in M$.
\end{abstract}

\section{Introduction}

It is a recurrent phenomenon in the history of dynamical systems that the chaotic behavior takes a primary role. 
This happened in Poincaré's study of homoclinic orbits, the Smale's Axiom A and the hyperbolicity of the non-wandering set of structurally stable diffeomorphisms and the robustness of the Lorenz attractor, to name a few examples.
The purpose of this article is to show the generic role played by continuum-wise expansivity in the context of homeomorphisms of compact surfaces with the $C^0$ topology. 

A chaotic topological property which is present in hyperbolic sets and the Lorenz attractor is expansivity. 
Recall that a homeomorphism $f$ of a metric space is \textit{expansive} if there is $\epsilon>0$ such that if $x\neq y$ then $\dist(f^i(x),f^i(y))>\epsilon$ for some $i\in\Z$.
In the context of surfaces this property is well understood. 
Indeed, Lewowicz \cite{Le89} and Hiraide \cite{Hi90} proved that expansive homeomorphisms of compact surfaces are conjugate to pseudo-Anosov diffeomorphisms, the two-sphere does not admit expansive homeomorphisms and for the two-torus they are conjugate to Anosov diffeomorphisms.
Some of their features are captured by a weaker notion introduced by Kato \cite{Kato93} called \textit{continuum-wise expansivity}. A \textit{continuum} is a non-empty, closed and connected subset of a given space.
We say that $f$ is \emph{cw-expansive} if there is $\epsilon>0$ such that
if $C$ is a continuum with more than one point 
then $\diam(f^i(C))>\epsilon$ for some $i\in\Z$.

Cw-expansivity is much weaker than expansivity and it is known to show several different features. 
For example, it is known that the sphere admits cw-expansive homeomorphisms \cite{Dend}. Also, there are cw-expansive diffeomorphisms with wandering points which are robustly cw-expansive in the $C^r$ topology, for $r\geq 2$ \cite{ArRobN}.
There are cw-expansive homeomorphisms of the genus 2 surface with a point whose local stable set is connected but not locally connected \cite{ArAnomalous}.
In \cite{AAV} it is shown that such systems may have infinitely many fixed points.
This zoo of examples on surfaces indicates that cw-expansivity is scattered in the set of homeomorphisms of surfaces. The main result of this paper, Theorem \ref{thmCwExpDensos} shows that cw-expansivity is $C^0$ dense for surfaces. But we have two stronger corollaries. 
We show that generic surface homeomorphism are extensions of cw-expansive systems, as explained in detail in the Corollaries \ref{corGenACEsup} and \ref{corGenACEsup2}.

Contents of the paper. In \S\ref{secDecompositions} we study general properties of decompositions, mainly for two-dimensional disks. We introduce the notion of \textit{cw-transversality} (continuum-wise transversality) by requiring a totally disconnected intersection between the objects.
When these objects are local stable and unstable continua, the cw-transversality allows us to conclude cw-expansivity. The main result of this section is Theorem \ref{thmPertCercaBorde} which proves that cw-transversality is a generic property for decompositions whose plaques are meagre (have empty interior). The statement gives a precise meaning for this.

In \S\ref{secSysDyn} we recall and develop some concepts of dynamical systems concerning cw-expansivity and almost cw-expansivity. For an arbitrary space we construct a natural $G_\delta$ subset of homeomorphisms which are almost cw-expansive. If the cw-expansive homeomorphisms are dense, then this $G_\delta$ subset is dense. 
We study some examples. In particular we show that this $G_\delta$ 
set may be dense even if the space admits no cw-expansive homeomorphism.

The main results of the paper are stated and proved in \S\ref{secDynSurf}. 

\section{Decompositions}
\label{secDecompositions}
This section is related to the theory developed in \cite{Dend} concerning 
decompositions as local charts of topological foliations to model the distribution of stable and unstable continua of a given homeomorphism.

Let $(X,\dist)$ be a compact metric space.
Let $\C(X)$ be the space of subcontinua of $X$ with the Hausdorff distance.
A \textit{decomposition} of $X$ is a partition map $Q\colon X\to \C(X)$
which is upper semicontinuous, \textit{i.e.} given $x\in X$ and $\epsilon>0$ there is $\delta>0$ such that if $\dist(y,x)<\delta$ then $Q(y)\subset B_\epsilon (Q(x))$.

The \textit{diameter} of a decomposition is
$\diam(Q)=\sup_{x\in X} \diam(Q(x))$.
We say that a decomposition $Q$ is
\textit{meagre} if each \textit{plaque} $Q(x)$ has empty interior.

\subsection{Restrictions and extensions}
Suppose that $Y\subset X$ is a closed subset.
For $Q$ a decomposition of $X$ define the \textit{restricted decomposition} $Q|_Y$ of $Y$ as
$$Q|_Y(y)=\comp_y(Q(y)\cap Y),\text{ for all }y\in Y$$
where $\comp_y(Z)$ is the connected component of $Z$ that contains $y$.

Suppose that $P$ is a decomposition of the subset $Y\subset X$. 
In this case we will say that $P$ is a \textit{partial decomposition} of $X$.
Define the \textit{singleton extension} of $P$ as the partition $P_{se}$ of $X$ given by $P_{se}(x)=P(x)$ for all $x\in Y$ and $P_{se}(x)=\{x\}$ otherwise.

\begin{prop}
If $P$ is a partial decomposition of $X$ then $P_{se}$ is a decomposition of $X$.
If in addition, $X$ has no isolated point and $P$ is meagre then its singleton extension $P_{se}$ is meagre.
\end{prop}
\begin{proof}
Suppose that $P$ is defined in the closed subset $Y\subset X$.
To show that $P_{se}$ is semicontinuous consider $x\in X$ and $\epsilon>0$ given. 
If $x\notin Y$, as $Y$ is closed there is $\delta\in(0,\epsilon)$ such that $B_\delta(x)$ is disjoint from $Y$. Thus, $\dist(y,x)<\delta$ implies $P_{se}(y)=\{y\}\subset B_\epsilon(x)=B_\epsilon(P_{se}(x))$.
If $x\in Y$ then there is $\delta\in (0,\epsilon)$ from the semicontinuity of $P$ over $Y$. 
Take $y\in B_\delta(x)$. If $y\in Y$ then $P(y)\subset B_\epsilon(P(x))$ from the semicontinuity over $Y$. 
If $y\notin Y$ then $P_{se}(y)=\{y\}\subset B_\epsilon(P_{se}(x))$ since $\delta<\epsilon$.

Assume that $P$ is meagre and $X$ has no isolated point. For $x\notin Y$ we have that $P_{se}(x)=\{x\}$ has empty interior because $X$ has no isolated point.
For $x\in Y$, if $U\subset P_{se}(x)$ is open then $U\subset P(x)$. Since $P$ is meagre, $U$ is empty and the proof ends.
\end{proof}

\subsection{Cw-transversality}
Given a compact metric space $Y$ and a continuous map $h\colon Y\to X$ define the \textit{pullback decomposition}
$h^*Q$ of $Y$ as
$$h^*Q(y)=\comp_y[h^{-1}(Q(h(y)))]$$
for all $y\in Y$.


\begin{df}
An arc is \textit{cw-transverse} to a decomposition if
it intersects each plaque in a totally disconnected subset.
\end{df}

\begin{rmk}
\label{rmkCwTrans}
 If $\alpha\colon [0,1]\to X$ is a parametrization of an arc (\textit{i.e.} it is continuous and injective)
 then $\alpha$ is cw-transverse to a decomposition $Q$ if and only if
 $\diam(\alpha^*(Q))=0$.
\end{rmk}

For $B\subset\R^n$ and $p\in B$ let $\cont_p([0,1],B)$ be the space of continuous functions $\phi\colon [0,1]\to B$ such that
$\phi(0)=\phi(1)=p$ with the $C^0$ topology.
The next result will be applied for $X$ a two-dimensional disk ($n=1$ in the statement).

\begin{prop}
\label{propCurvaCwTransv}
For $n\geq 1$ let $B\subset \R^n$ be a closed Euclidean ball centered at the origin $0\in\R^n$. 
If $Q$ is a meagre decomposition of $X=[0,1]\times B$ then a generic $\phi\in\cont_0([0,1],B)$ is cw-transverse to $Q$.
\end{prop}

\begin{proof}
For $\phi\in\cont_0([0,1],B)$ define $\tilde\phi(x)=(x,\phi(x))$.
For $\epsilon>0$ let
$$
U_\epsilon=\{\phi\in \cont_0([0,1],B):\diam(\tilde\phi^*Q)<\epsilon\}.
$$
We will show that $U_\epsilon$ is open and dense in $\cont_0([0,1],B)$.
We start proving that the complement of each $U_\epsilon$ is closed.
 Take a convergent sequence $\phi_k\in \cont_0([0,1],B)$ such that
 $\diam(\tilde\phi_k^*Q)\geq \epsilon$ for all $k\geq 1$ and $\phi_k\to\psi\in \cont_0([0,1],B)$.
 This implies that there is a sequence $x_k\in [0,1]$ such that
 \[
  \diam(\tilde\phi_k^*Q(x_k))
  =\diam(\comp_{x_k}[\tilde\phi_k^{-1}(Q(\tilde\phi_k(x_k)])
  \geq \epsilon\text{ for all }k\geq 1.
 \]
 Taking a subsequence we can suppose that $x_k\to y\in[0,1]$. The semicontinuity of the maps involved in the previous equation implies that 
$\diam(\tilde\psi^*Q(y))\geq\epsilon$ and $U_\epsilon$ is open.

We will show that $U_\epsilon$ is dense in $\cont_0([0,1],B)$.
Take any $\phi\in \cont_0([0,1],B)$.
By definition, $\tilde\phi^*(Q)$ is a decomposition of the interval $[0,1]$ into closed subintervals. Thus, only a finite number of such subintervals can have diameter greater than $\epsilon$. Suppose that for $x\in[0,1]$ we have 
$\diam(\tilde\phi^*Q(x))\geq\epsilon$. 
This means that for some subinterval $[y,z]\subset [0,1]$, $y\leq x\leq z$, the curve 
$\tilde\phi([y,z])\subset X$ is contained in the plaque $Q(x,\phi(x))$.
Let $u=(y+z)/2$. Since $Q$ is meagre the plaques of $Q$ have empty interior and we can take $p\in B$ close to $\phi(u)$ and so that $(u,p)\notin Q(x,\phi(x))=Q(u,\phi(u))$.
Therefore, we can define $\phi_1\in C_0([0,1],B)$, close to $\phi$, in such a way that 
$\phi_1(a)=\phi(a)$ for $a\in[0,y]\cup [z,1]$ and $\phi_1(u)=p$. 
For $\phi_1$ we have that neither $\tilde\phi_1([y,u])$ nor 
$\tilde\phi_1([u,z])$ are contained in a plaque of $Q$.
Repeating this procedure finitely many times we obtain a map in $U_\epsilon$ which is close to the given $\phi$.

Therefore, $G=\cap_{j\geq 1} U_{1/j}$ is a residual subset of
$\cont_0([0,1],B)$.
By Remark \ref{rmkCwTrans}, each $\phi\in G$ is cw-transverse to $Q$ and the proof ends.
\end{proof}


\subsection{Decompositions of disks}
Define the sets:
\begin{itemize}
 \item the disk $D=\{(x_1,x_2)\in\R^2:x_1^2+x_2^2\leq 1\}$,
 \item the circle $\alpha_r=\{(x_1,x_2)\in\R^2:x_1^2+x_2^2=r^2\}$ for $0<r<1$,
 \item the annulus $A_r=\{(x_1,x_2)\in\R^2:r^2\leq x_1^2+x_2^2\leq 1\}$ for $0<r<1$.
\end{itemize}
Let $\pi\colon A_r\to\partial D$ be the radial projection.


\begin{lem}
\label{lemArrimoBorde}
Suppose that $\partial D$ is cw-transverse to the decomposition $Q$.
For all $\epsilon>0$ there is $\delta>0$ such that if
$1-\delta\leq r<1$
then
$\diam(\pi(Q|_{A_r}(x)))<\epsilon$ for all $x\in A_r$. This diameter of subsets of $\partial D$ is defined from the arclength of the circle.
\end{lem}

\begin{proof}
 Arguing by contradiction, suppose that for some $\epsilon>0$ there are a sequence $r_n\to 1$, $r_n<1$ and $x_n\in A_{r_n}$ such that
\begin{equation}
 \label{ecuContBorde}
 \diam(\pi(Q|_{A_r}(x)))\geq\epsilon\text{ for all }n\geq 1.
\end{equation}
 Taking a subsequence we can assume that $x_n\to y\in\partial D$ and
$Q|_{A_{r_n}}(x_n)\to C$, where $C$ is a subcontinuum of $\partial D$.
From \eqref{ecuContBorde} we have that $\diam(C)\geq\epsilon>0$.
The semicontinuity of $Q$ gives us that $C\subset Q(y)$.
This contradicts that $\partial D$ is cw-transverse to $Q$ and the proof ends.
\end{proof}

If $P,Q$ are decompositions of $D$ define the \textit{intersection decomposition} $P\cap Q$ of $D$ as
\[
[P\cap Q] (x)=\comp_x[P(x)\cap Q(x)]
\]
for all $x\in D$.

\begin{lem}
\label{lemPertCercaBorde}
 If $P,Q$ are cw-transverse to $\partial D$
 then there is  a homeomorphism $h\colon D\to D$ that fixes the boundary
 and
 \[
  [(h^*P)\cap Q](x)\cap \partial D=\emptyset \text{ for all } x\in\interior(D).
 \]
\end{lem}

\begin{proof}
Fix $\epsilon_*>0$ small. Let $\epsilon_n<\epsilon_*/2$ be a positive sequence with $\epsilon_n\to 0$.
 From Lemma \ref{lemArrimoBorde} applied to $P$ and $Q$
 take the corresponding values of $\delta_n$ (to have a common value of such $\delta$ take the minimum of the corresponding values for $P$ and $Q$), which we also assume that $\delta_n\to 0$.
 Let $r_n=1-\delta_n$.

 The homeomorphism $h$ of $D$ will be defined in polar coordinates by the formula $h(r,\theta)=(r,\theta+s(r))$, where $s\colon [0,1]\to [0,\epsilon_*]$ is continuous, $s(0)=s(1)=0$ and $s(r_n)>2\epsilon_n$
 for all $n\geq 1$.

 Arguing by contradiction, suppose that there is $x\in\interior D$ such that
$$[(h^*P)\cap Q](x)\cap \partial D\neq\emptyset.$$
We have that $C_1=[(h^*P)\cap Q](x)$ is a continuum contained in plaques of $Q$ and $h^*P$, containg $x$ and intersectting $\partial D$.
As $x$ is interior to $D$, there is $r_n>\|x\|$.
We can take a subcontinuum $C_2\subset C_1$ such that $C_2\subset A_{r_n}$ and $C_2$ instersects both boundaries $\partial D$ and $\alpha_r$.
Take $y\in C_2\cap \alpha_{r_n}$ and $z\in C_2\cap \partial D$.
On the one hand we have that $C_2\subset Q(z)$.
On the other hand, since $h$ fixes the boundary of the disk we have $h(z)=z$ and
$C_2\subset h^*P(z)$.

Since $C_2\subset A_{r_n}\cap Q(z)$, $z\in C_2\cap\partial D$ and $y\in C_2 \cap \alpha_{r_n}$ we can apply Lemma \ref{lemArrimoBorde} to obtain $\dist(\pi(y),z)\leq \diam(\pi(C_2))<\epsilon_n$.
Also $y\in h^*P(z)$
and $h(y)\in P(z)$;
thus we can apply Lemma \ref{lemArrimoBorde} to obtain $\dist(\pi(h(y)),z)<\epsilon_n$.
Therefore, $\dist(\pi(y),\pi(h(y))<2\epsilon_n$.
If $y$ has polar coordinates $(r_n,\theta_y)$ then the coordinates of $h(y)$ are
$(r_n,\theta_y+s(r_n))$. This contradicts that $s(r_n)>2\epsilon_n$ and the proof ends.
\end{proof}

\begin{df}
 We say that the decompositions $P,Q$ of $X$ are \textit{cw-transverse} if $[P\cap Q](x)=\{x\}$ for all $x\in X$.
\end{df}
Let $\homeo(D)$ be the space of homeomorphisms of $D$ with the $C^0$ topology.
\begin{thm}
\label{thmPertCercaBorde}
 If $P,Q$ are meagre decompositions of $D$ and cw-transverse to $\partial D$
 then for a generic homeomorphism $h\in\homeo(D)$ that fixes the boundary
 the decomposition $h^*P$ is cw-transverse to $Q$.
\end{thm}

\begin{proof}
 Let $\homeo_\partial(D)$ be the space of homeomorphisms $h\colon D\to D$ with $h(x)=x$ for all $x\in\partial D$ endowed with the $C^0$ topology.
 For $\epsilon>0$ define
 \[
 U_\epsilon=\{h\in\homeo_\partial(D):\diam[(h^*P)\cap Q]<\epsilon\}.
 \]
Let us show that the complement of $U_\epsilon$ is closed in $\homeo_\partial(D)$.
Take $h_k\to h$ with $h_k\notin U_\epsilon$.
That is, there is a sequence $x_k\in D$ such that
if
$$
C_k=\comp_{x_k}([h_k^*P](x_k)\cap Q(x_k))
$$
then
\[
 \diam(C_k)\geq \epsilon \text{ for all }k\geq 1.
\]
As $D$ is compact we can assume that $x_k\to y\in D$.
Also $\C(D)$ is compact with the Hausdorff distance, thus we assume that $C_k$ converges to a continuum $C\subset D$ with $\diam(C)\geq \epsilon$, and due to the semicontinuity of $Q$ we have that $C\subset Q(y)$.
We know that $h_k(C_k)\subset P(h_k(x_k))$, which implies $h(C)\subset P(h(y))$.
Then $C\subset h^*P(y)$, $h\notin U_\epsilon$ and $U_\epsilon$ is open $\homeo_\partial(D)$.

To prove that $U_\epsilon$ is dense in $\homeo_\partial(D)$ take any $h\in \homeo_\partial(D)$ and $\delta>0$.
Take $\gamma\in(0,\epsilon)$ such that $\dist(x,y)<\gamma$ implies $\dist(h(x),h(y))<\delta$ for all $x,y\in D$.
Consider a triangulation of $D$ made of disks $D_1,\dots,D_l$ of diameter less than $\gamma$.
By Proposition \ref{propCurvaCwTransv} we can assume that $\partial D_j$ is cw-transverse to $h^*P$ and $Q$ for all $j=1,\dots,l$.
For each $j=1,\dots,l$ let $h_j\colon D_j\to h(D_j)$ be the restriction of $h$
and define $P_j=h_j^*(P|_{h(D_j)})$ and $Q_j=h_j^*(Q|_{D_j})$.
Applying Lemma \ref{lemPertCercaBorde} to the decompositions $P_j,Q_j$ of $D_j$ we obtain a homeomorphism
$g_j$ of $D_j$ that fixes $\partial D_j$ and satisfies
\begin{equation}
 \label{ecuCondgj}
  [(g_j^*P_j)\cap Q_j](x)\cap \partial D_j=\emptyset \text{ for all } x\in\interior(D_j).
\end{equation}
Define $h'\in\homeo_\partial(D)$ as
$h'(x)=h(g_j(x))$ for $x\in D_j$.
Since $g_j(x)\in D_j$, we have that $h(x), h'(x)\in h(D_j)$.
As $\diam(h(D_j))<\delta$ we have $\dist_{C^0}(h',h)<\delta$.

To prove that $h'\in U_\epsilon$ note that $\diam(D_j)<\gamma<\epsilon$.
We need to prove that
\begin{equation}
 \label{ecuDiamComp}
\diam[(h'^*P)\cap Q](x)
=\diam \comp_x(h'^{-1}(P(h'(x)))\cap Q(x))
<\epsilon.
\end{equation}
If $x\in D_j$ then
\[
 h'^{-1}(P(h'(x)))=
 g_j^{-1}\circ h^{-1}(P(h\circ g_j(x))).
\]
Since $\partial D_j$ is cw-transverse to $h^*P$ and $Q$ we can assume that $x$ is in the interior of $D_j$ (possibly another disk $D_j$ of the triangulation).
Let $z=h\circ g_j(x)$.
Arguing by contradiction suppose that \eqref{ecuDiamComp} is not true, which implies
\[
 \comp_x(g_j^{-1}\circ h^{-1}(P(z))\cap Q(x)) \text{ intersects the boundary of } D_j.
\]
To apply the condition \eqref{ecuCondgj}
notice that
\[
\begin{array}{rl}
(g_j^*P_j)(x)& =g_j^{-1}(P_j(g_j(x)))=g_j^{-1}([h_j^*(P|_{h(D_j)})](g_j(x)))\\
& = g_j^{-1}\circ h_j^{-1}([P|_{h(D_j)}](h_j\circ g_j(x))).\\
\end{array}
\]
Since $z=h_j\circ g_j(x)$ we have
\[
\begin{array}{rl}
(g_j^*P_j)(x)
&=g_j^{-1}\circ h_j^{-1}([P|_{h(D_j)}](z))\\
&=g_j^{-1}\circ h_j^{-1}(\comp_z(P(z)\cap h(D_j))).\\
\end{array}
\]
Condition \eqref{ecuCondgj} now implies that the continuum
\[
[(g_j^*P_j)\cap Q_j](x)=\comp_x[[g_j^{-1}\circ h_j^{-1}(\comp_z(P(z)\cap h(D_j)))]\cap Q_j(x)]
\]
is disjoint from $\partial D_j$. This contradiction proves that $h'\in U_\epsilon$.

Finally the $G_\delta$ dense subset we are looking for can be defined as $\cap_{n\geq 1} U_{1/n}$.
\end{proof}

\section{Dynamical Systems}
\label{secSysDyn}
As before, let $\homeo(M)$ be the space of homeomorphisms of $M$ with the $C^0$ topology.
The homeomorphisms $f\in\homeo(M)$ and $g\in\homeo(N)$ are \textit{semiconjugate} if
there is $\pi\colon M\to N$, continuous and onto, such that $\pi f=g\pi$.
In this case $\pi$ is a \textit{semiconjugacy}.
In what follows we collect some dynamical properties that are transferred from $f$ to $g$.

\begin{prop} [\cite{AAV}*{Proposition 5.4}]
\label{propCocCwExp}
Suppose that $M, N$ are compact metric spaces, $f\in\homeo(M)$, $g\in\homeo(N)$, $\pi\colon M\to N$ is a semiconjugacy and $\pi^{-1}(y)$ is totally disconnected for all $y \in N$.
 Under these conditions, if $f$ is cw-expansive then $g$ is cw-expansive.
\end{prop}

We say that $\Lambda\subset M$ is $\pi$-\textit{saturated} if $\pi^{-1}(\pi(\Lambda))=\Lambda$.
If $\pi\colon M\to N$ is open and $\Lambda\subset M$ is $\pi$-saturated then 
$\pi\colon \Lambda\to\pi(\Lambda)$ is open.
This and the previous proposition implies the next result.

\begin{cor}
\label{corCocCwExp}
Assuming the hypothesis of the previous proposition, if $\pi$ is open and $\Lambda\subset M$ is $\pi$-saturated, closed, $f$-invariant and $f\colon\Lambda\to\Lambda$ is cw-expansive then $g\colon\pi(\Lambda)\to \pi(\Lambda)$ is cw-expansive.
\end{cor}

A closed $f$-invariant subset $X\subset M$ is $f$-\textit{isolated} if there is 
a neighborhood $V$ of $X$ such that $f^n(x)\in U$ for all $n\in\Z$ implies $x\in X$. 
In this case, $V$ is an \textit{isolating neighborhood} of $X$.

\begin{prop}
\label{propCocienteIso}
Suppose that $M, N$ are compact metric spaces, $f\in\homeo(M)$, $g\in\homeo(N)$, $\pi\colon M\to N$ is an open semiconjugacy and $\Lambda\subset M$ is $\pi$-saturated and $f$-isolated then $\pi(\Lambda)$ is $g$-isolated.
\end{prop}

\begin{proof}
To prove that $\pi(\Lambda)$ is $g$-isolated let $V$ be an isolating neighborhood of $\Lambda$. 
Since $\pi$ is open, $\pi(V)$ is a neighborhood of $\pi(\Lambda)$. 
By the continuity of $\pi$ and the fact that $\Lambda$ is $\pi$-saturated, we can take an open set $U\subset N$ such that $\pi(\Lambda)\subset U\subset \pi(V)$ and $\pi^{-1}(U)\subset V$. 
In this way, $U$ is an isolating neighborhood of $\pi(\Lambda)$ and the proof ends.
\end{proof}

\subsection{Almost cw-expansivity}
A map $\pi\colon M\to N$ is \textit{monotone} if $\pi^{-1}(y)\subset M$ is connected for all $y\in N$.
Define
\[
 \mesh(\pi)=\sup_{y\in N} \diam(\pi^{-1}(y)).
\]

\begin{df}
We say that $f\in\homeo(M)$ is \emph{almost cw-expansive} if for all $\epsilon>0$
there are a compact metric space $N$, $g\in\homeo(N)$ cw-expansive and a monotone semiconjugacy
$\pi\colon M\to N$ such that $\mesh(\pi)<\epsilon$.
\end{df}

\begin{rmk}
If $f$ is cw-expansive, given any $\epsilon>0$ as in the previous definition 
we can take $N=M$, $g=f$ and $\pi$ as the identity.
Thus, every cw-expansive homeomorphism is almost cw-expansive.
\end{rmk}

\begin{prop}
\label{propCharACE}
For $f\in\homeo(M)$ the following are equivalent:
\begin{enumerate}
 \item $f$ is almost cw-expansive,
 \item for all $\epsilon>0$ there is an $f$-invariant equivalence relation $\sim$ with connected and upper semicontinuous clasess, such that each class has diameter less than $\epsilon$ and whose quotient homeomorphism is cw-expansive.
\end{enumerate}
\end{prop}

\begin{proof}
 To prove that $2 \Rightarrow 1$ notice that the quotient space can be taken as $N$ and the projection as the semiconjugacy.
 For $1 \Rightarrow 2$, given such semiconjugacy $\pi$ define the equivalence relation on $M$ as $x\sim y$ if $\pi(x)=\pi(y)$.
The remaining details of the proof are left to the reader.
\end{proof}

\subsection{A residual set}
\label{secResidual}
In the set of homeomorphisms of $M$ we consider the $C^0$ topology.
Let us show a natural $G_\delta$ set (a countable intersection of open sets) of almost cw-expansive homeomorphisms for any compact metric space $M$.
Given $\expc>0$ define the set
\[
E_\expc=\left\{
f\in\homeo(M):
\sup_{i\in\Z}\diam f^i(C)\leq\expc\Rightarrow
\sup_{i\in\Z}\diam f^i(C)<\frac\expc 2,
\forall C\in \C(M)
\right\}.
\]
For $f\in E_\expc$ define the equivalence relation $x\sim y$ if there is a continuum $C\subset M$ such that $x,y\in C$ and $\diam(f^i(C))\leq\expc$ for all $i\in\Z$.
By \cite{AAV}*{Theorem 2.12} (see also \cite{FG}*{Theorem 2.11}) the homeomorphism $g$ induced in the (compact metric) quotient space $N=M/\sim$
is cw-expansive.
Also, the quotient map $\pi$ is monotone (by definiton) and $\mesh(\pi)< \expc/2$ (by the definition of the set $E_\alpha$).

By \cite{AAV}*{Theorem A} (essentially \cite{FG}*{Lemma 2.1}) we have that each $E_\expc$ is open in the space of 
homeomorphisms of $M$ with the $C^0$ topology. 
%
For $n\in\N$ define
\begin{equation}
\label{ecuFn}
F_n=\cup_{0<\expc<1/n} E_\expc.
\end{equation}
Since each $E_\expc$ is open, we have that $F_n$ is open and
$$G=\cap_{n\in\N} F_n$$
is $G_\delta$ and from the construction and previous remarks, each $f\in G$ is almost cw-expansive.
\begin{rmk}
\label{rmkCwDenseAlmcwexpGen}
If the set of cw-expansive homeomorphisms of $M$ is dense then each $F_n$ is dense and $G$ is a dense $G_\delta$ subset. Thus, for such a space the property of almost cw-expansivity is generic.
\end{rmk}

Of course, in general the set $G$ could be empty, depending on the space $M$. Let us show some examples.

\subsection{Examples}

It is known that the arc \cite{Bryant54}, the circle \cite{JU} and dendrtites \cite{Kato902} do not admit cw-expansive homeomorphisms.
If $M$ is an arc or ar circle and $\epsilon<\diam(M)$ then any quotient as in Proposition \ref{propCharACE} is again homeomorphic to $M$.
For $M$ a dendrite such quotient gives again a dendrite.
Thus, theses spaces do not admit almost cw-expansive homeomorphisms.

Let us give an example of a space $M$ with almost cw-expansive homeomorphisms but not supporting cw-expansivity. For each integer $n\geq 1$ consider 
two-dimensional spheres in three-dimensional space, tangent to the $x=0$
plane at the origin and with radius $1/n$
$$ S_n=\{(x,y,z)\in\R^3:(x-1/n)^2+y^2+z^2=1/n^2\} $$
and define
$$X=\cup_{n\geq 1} S_n.$$

\begin{prop}
 The space $X$ admits no cw-expansive homeomorphism.
\end{prop}

\begin{proof}
If $f$ is a homeomorphism of the space $X$ then it fixes de origin an the spheres are permuted. 
If every sphere is periodic, \textit{i.e.} for each $n\geq 1$ there is $k\geq 1$ such that $f^k(S_n)=S_n$, 
considering that for each $\epsilon>0$ there is just a finite number of spheres with diameter larger than $\epsilon$, we conclude that there must be a cycle of small spheres. This implies that $f$ cannot be cw-expansive.

If $f$ has a wandering sphere, \textit{i.e.} a sphere $S_m$ such that $f^k(S_m)\cap S_m=\{0\}$ for all $k>0$, 
then it is easy to see that $\diam(f^k(S_m))\to 0$ as $k\to\pm\infty$. Thus, we can take a small subcontinuum of $S_m$, with small iterates. This also implies that $f$ cannot be cw-expansive.
\end{proof}

In the next result we suppose that each sphere is invariant with an arbitrary cw-expansive dynamics, for instance we can assume that $f\colon S_n\to S_n$ is conjugate to a pseudo-Anosov map with spines.
This homeomorphism is not cw-expansive in $X$ but we will see that it is almost cw-expansive.

\begin{prop}
 If $f\in\homeo(X)$ is so that $f\colon S_n\to S_n$ is cw-expansive then
 $f$ is almost cw-expansive.
\end{prop}

\begin{proof}
For $n>1$ let $\pi_n\colon X\to Y_n=\cup_{i=1}^{i=n}S_n$ be defined as 
$\pi_n(x)=x$ if $x\in S_m$ with $1\leq m\leq n$ and 
$\pi_n(x)=0$ if $x\in S_m$ with $m>n$. 
That is, $\pi_n$ collapses all the spheres $S_m$ with $m>n$ to the origin.
On the remaining spheres define $g_n\colon Y_n\to Y_n$ as the restriction of $f$, 
\textit{i.e.} $g_n(x)=f(x)$ for all $x\in Y_n$. 
Since $f$ is cw-expansive on each $S_n$, we have that $g_n$ is cw-expansive. 
By construction we have $\mesh(\pi_n)\to 0$ as $n\to\infty$. Since $\pi_n$ is a semiconjugacy for $f$ and $g_n$ we conculde that $f$ is almost cw-expansive.
\end{proof}

\begin{rmk}
From the result that will follow in this paper, it can be proved that 
for this space $X$ the set of almost cw-expansive homeomorphisms is dense in the $C^0$ topology.
\end{rmk}

\section{Dynamics on surfaces}
\label{secDynSurf}

In this section we prove the main results of the article.
	Let us first give some ideas about the proofs. To prove the density of cw-expansive homeomorphisms we use that structurally stable diffeomorphisms are $C^0$ dense. It is known that a cw-expansive homeomorphism cannot have stable points, thus we need a small perturbation around attracting and repelling periodic points. For this purpose we combine two well known constructions. On the one hand, we consider the derived from Anosov diffeomorphism, that gives rise to an expanding attractor on the two-torus. 
On the other hand, we have the antipodal quotient of the two-torus giving rise to a pseudo-Anosov map with 1-prong singularities on the two-sphere. From these examples we obtain an expanding attractor with 1-prong singularities which is contained in a disk. Each attracting and repelling periodic orbit is replaced by this expanding attractor, which is cw-expansive. At this point each stable and unstable manifold is one-dimensional. To obtain cw-expansivity we need another perturbation such that the intersections of local stable and unstable sets are totally disconnected. Such results on perturbations were collected in \S\ref{secDecompositions}, the key is Theorem \ref{thmPertCercaBorde}.
To prove the results about genericity of almost cw-expansivity we mainly apply known results from \cite{AAV}.

In this section we assume that $M$ is a smooth, compact surface without boundary.

\subsection{Derived from pseudo-Anosov homeomorphism}
\label{secDPA}

Let us recall the definition of the derived from Anosov (DA) diffeomorphism \cite{Robinson}*{\S 7.8}.
It starts with the linear Anosov diffeomorphism of the two-dimensional torus $\T^2$ given by $g(x,y)=(2x+y,x+y)$.
If $v^s$ and $v^u$ are the stable and unstable eigenvectors we use the coordinates $u_1 v^u+u_2 v^s$ in a small ball $U$ of radius $r_0>0$ centered at the origin.
Let $\delta\colon\R\to[0,1]$ be a bump function such that $\delta(x)=0$ for $x\geq r_0$ and
$\delta(x)=1$ for $x\leq r_0/2$.

Consider the vector field on the torus defined by
$$v(u_1,u_2)=(0,u_2\delta(\sqrt{u_1^2+u_2^2}))$$
and $v=0$ outside $U$.
\begin{rmk}
Note that $v(-u_1,-u_2)=-v(u_1,u_2)$ for all $u_1,u_2$.
\end{rmk}

If $\phi$ is the flow induced by $v$ then de DA is $f_{DA}=\phi_\tau\circ g$ for suitable $\tau$.
We do not develope the details, for which we refer the reader to \cite{Robinson}, but, for example,
$\tau$ has to be large enough to make the origin a source.
From the construction it is clear that
the stable foliation of the Anosov diffeomorphism $g$ is also invariant for $f_{DA}$.

\begin{thm}[\cite{Robinson}*{Theorem 8.1}]
 The derived from Anosov diffeomorphism $f_{DA}$ has non-wandering set $\Omega(f_{DA})=\{p_0\}\cup\Lambda$ where $p_0$ is a source
 and $\Lambda$ is a hyperbolic expanding attractor of topological dimension one. The map $f$ is transitive on $\Lambda$
 and the periodic points are dense in $\Lambda$.
\end{thm}

\begin{rmk}[Antipodal relation and the quotient sphere]
\label{rmkCondPi}
On $\T^2$ we have the antipodal equivalence relation $p\sim -p$. The quotient space is a two-dimensional sphere that we will denote as $\Sph^ 2$. 
Let $\pi\colon\T^2\to\Sph^2$ be the quotient map. The map $\pi$ has four branching points that can be characterized by $p\sim-p$. These point are obtained by solving $p-(-p)=2p\in\Z^2$, or in coordinates $p=(x,y)$, $x,y\in \frac12\Z$. This gives 4 solutions $p_0=(0,0)$, $p_1=(\frac12,0)$, $p_2=(0,\frac12)$ and $p_3=(\frac12,\frac12)$. 
In particular we see that for any $u\in\Sph^2$, the set $\pi^{-1}(u)$ has at most 2 points and each $\pi^{-1}(u)$
is totally disconnected. We also need to remark that $\pi$ is open, which is easy to prove.
\end{rmk}

 Since $X(-p)=-X(p)$ we have that $\phi_\tau(-p)=-\phi_\tau(p)$.
 Thus, $f_{DA}(-p)=-f_{DA}(p)$ for all $p\in \T^2$.
That is, we have an induced map in the quotient that will be denoted as 
$f_{PDA}\colon\Sph^2\to\Sph^2$. 
It will be called \textit{pseudo-DA} homeomorphism of the two-dimensional sphere.
For $p\in\T^2$ let $\tilde p=\pi(p)$.
For the branching points of $\pi$ we have:
$\tilde p_0$ is a repelling fixed point
 and $\tilde p_1$, $\tilde p_2$, $\tilde p_3$ define a period 3 orbit for $f_{PDA}$.
 The points $\tilde p_1,\tilde p_2,\tilde p_3$ have 1-prong stable and unstable arcs.
The fixed point $\tilde p_0$ in the sphere is a repelling fixed point and the non-wandering set of $f_{PDA}$ has another piece $\tilde\Lambda\subset\Sph^2$ containing the perido 3 orbit.

\begin{prop}
\label{propPDAcwexpIso}
For $f_{PDA}$ the piece $\tilde\Lambda$ is cw-expansive and isolated.\footnote{In fact it is positively cw-expansive and an attractor.}
\end{prop}

\begin{proof}
Notice that $\Lambda$ is a hyperbolic set for $f_{DA}$ and consequently it is expansive and cw-expansive.
Since the semiconjugacy $\pi$ satisfies the conditions explained in Remark \ref{rmkCondPi} we can apply Corollary \ref{corCocCwExp} to conclude that  
$f_{PDA}$ is cw-expansive on $\tilde\Lambda$.
Analogously, by Proposition \ref{propCocienteIso} we have that $\tilde\Lambda$ is isolated for $f_{PDA}$.
\end{proof}

\subsection{Diffeomorphisms}
\label{secDiff}
This brief section is devoted to summarize known results about diffeomorphisms that will be useful in what follows.
The first result allows us to perturb a homeomorphism to obtain a diffeomorphism.

\begin{thm}[Munkres-Whitehead]
\label{thmMuWh}
Let $M$ and $N$ be $n$-dimensional differentiable manifolds, $n\leq 3$.
If $f\colon M\to N$ is a homeomorphism then $f$ may be $C^0$ approximated by a diffeomorphism of $M$ onto $N$.
\end{thm}

Theorem \ref{thmMuWh} was proved by J. Munkres \cite{Mun} and J.H.C. Whitehead \cite{Wh}*{Corollary 1.18}; and we learnt it from \cite{Pil}*{Theorem 0.1.1}.

The next result allows us to obtain a structurally stable diffeomorphism.
Let $\Diff^r(M)$ be the space of $C^r$ diffeomorphisms of $M$ with the $C^r$ topology.
We say that $f\in\Diff^r(M)$ is \textit{structurally stable} if there exists a neighborhood $U_f$ of $f$ in $\Diff^r(M)$ such that given $g\in U_f$ there exists $h\in\homeo(M)$ such that $hf=gh$.

\begin{thm}[Shub]
\label{thmShub}
Let $M$ be a $C^\infty$ compact manifold without boundary and $1\leq r\leq\infty$. Then the structurally stable diffeomorphisms are dense in $\Diff^r(M)$ with the $C^0$ topology.
\end{thm}

Theorem \ref{thmShub} was proved by M. Shub \cite{Shub} following S. Smale \cite{Smale}.
In fact both conclude the structural stability by proving the
Axiom A, the no-cycle condition and the transversality condition; and applying J. Robbin's Theorem \cite{Robbin}.
For our purposes we state: for $M$ a $C^\infty$ compact surface without boundary
every homeomorphism of $M$ can be $C^0$ approximated by a diffeomorphism satisfying
the Axiom A, the spectral decomposition of the non-wandering set \cite{Smale67}, the no-cycle condition and the transversality condition.


\subsection{Density and genericity on surfaces}

We state and prove the main results of the paper.

\begin{thm}
\label{thmCwExpDensos}
If $M$ is a closed surface then the set of cw-expansive homeomorphisms of $M$ is $C^0$ dense in $\homeo(M)$.
\end{thm}

\begin{proof}

From \S\ref{secDiff} we start with a diffeomorphism
$f_0$ satisfying Smale's Axiom A, the transversality condition on 
stable and unstable manifolds and the no cycle condition. We will $C^0$ approximate $f_0$ by a homeomorphism in $F_n$.

Suppose that $p\in M$ is an attracting periodic point of $f_0$ of period $k$.
Then, performing a suitable small $C^0$ perturbation of $f_0$ supported in a small neighborhood of the orbit of $p$
we can obtain a dynamics conjugate to the homeomorphism $f_{PDA}$ defined in \S\ref{secDPA} in the neighborhood of $p$.
Repeating this procedure around each attracting periodic orbit.
For repelling periodic orbits we do the same but considering the inverse $f_{PDA}^{-1}$.
In this way we obtain a homeomorphism $f_1$ which is $C^0$ close to $f_0$ and its stable and unstable manifolds are all one-dimensional.
Let us call \textit{singular basic sets} to the new pieces introduced by $f_1$ in the nonwandering set.
The other pieces, the basic set common to $f_0$ and $f_1$, are hyperbolic, thus expansive, cw-expansive and isolated.
The new pieces of $f_1$ are cw-expansive and isolated by Proposition \ref{propPDAcwexpIso}.
Thus, the non-wandering set of $f_1$ is a finite union of isolated cw-expansive sets.

In what follows we will approximate $f_1$ by a 
cw-expansive homeomorphism.
From what we know about the non-wandering set of $f_1$ we can take $\alpha>0$ such that if
\begin{equation}
\label{ecuContCChico}
\sup_{i\in\Z}\diam(f^i(C))\leq\alpha
\end{equation}
for a non-trivial continuum $C\subset M$ then
$C\cap\Omega(f_1)=\emptyset$.
Moreover, from the construction of $f_1$ we have that if $C$ satisfies \eqref{ecuContCChico} then there are $x,y\in\Omega(f_1)$
such that $C\subset W^u(x)\cap W^s(y)$.
From the no cycle condition we have that $x,y$ are in different basic sets of $f_1$.


We have that all the invariant manifolds (stable and unstable) of $f_1$ are one-dimensional; and they are transverse  if the points are associated to hyperbolic pieces of $f_1$ (those comming from $f_0$). This means that if such continuum $C$ exists then at least one of the points $x,y$
must belong to a singular basic set of $f_1$.

Let $U\subset B_\alpha(\Omega(f_1))$ be an isolating neighborhood of the non-wandering set of $f_1$.
It is a general fact (for any homeomorphism of a compact metric space) that there is $N>0$ such that each point of $M$ has at most $N$ iterates outside $U$.
Thus, to finish the proof it is sufficient to perform a perturbation supported on wandering points
near the singular attracting and repelling basic sets
to obtain cw-transversality of stable and unstable manifolds at every point of the surface.

Suppose that $\Lambda_{at}\subset M$ is a singular attracting set locally conjugate to $f_{PDA}$
and take a closed disc $V$ containing $\Lambda_{at}$ in its interior
and satisfying $f_1(V)\subset \interior V$.
Let $Q^s_{at}$ be the stable decomposition of $V$ into stable continua of $\Lambda_{at}$.
We can assume that $\partial V$ is cw-transverse to $Q^s_{at}$.
Let $\Lambda_1,\dots,\Lambda_k$ be the basic sets of $f_0$ with one-dimensional splitting. 
These pieces are shared with $f_1$. 

Each $\Lambda_i$ has a closed neighborhood $V_i$ where the unstable manifolds $W^u_{V_i}(\Lambda_i)$ make a partial decomposition of $V_i$.
The existence and semicontinuity of the these arcs is a consequence of the classical Invariant Manifold Theorem.
Let $Q^u_i$ be the unstable partial decomposition of $V_i$ by unstable arcs of $\Lambda_i$.
For each $n\geq 0$ and $1\leq i\leq k$, we have that $f^n(Q^u_i)$ is a partial decomposition of $f^n(V_i)$. 
Let $Q^u_{i,n}$ be the decomposition of $V_{at}$ (the neighborhood of the attractor $\Lambda_{at}$) whose plaques are the connected components of $f^n(Q^u_i)\cap V_{at}$ and extended by singletons to the whole $V_{at}$. 
By construction, each $Q^u_{i,n}$ is meagre.

By Proposition \ref{propCurvaCwTransv} we have that a generic perturbation of $\partial V_{at}$ is cw-transverse to $Q^s_{at}$ and $Q^u_{i,n}$. Thus, we assume that $\partial V_{at}$ is cw-transverse to these meagre decompositions, for all $n\geq 0$ and $1\leq i\leq k$.
Now, applying Theorem \ref{thmPertCercaBorde} we can perturb $f_1$ to obtain a homeomorphism $f_2$ such that $Q^s_{at}$ is cw-transverse to $Q^u_{i,n}$ 
for all $n\geq 0$ and $1\leq i\leq k$.
Repeating this procedure to the remaining singular attractors and repellers we obtain a perturbation which is cw-expansive and the proof ends.
\end{proof}

%

\begin{cor}
\label{corGenACEsup}
For a generic homeomorphism $f$ of the compact surface $M$ without boundary it holds that 
$f$ is almost cw-expansive.
\end{cor}

\begin{proof}
From Theorem \ref{thmCwExpDensos} we know that cw-expansive homeomorphisms are $C^0$ dense. Consequently, the sets $F_n$ defined in \eqref{ecuFn} are dense and we conclude that almost cw-expansivity is generic as explained in 
Remark \ref{rmkCwDenseAlmcwexpGen}.
\end{proof}

\begin{cor}
\label{corGenACEsup2}
For a generic homeomorphism $f$ of the compact surface $M$ without boundary it holds that 
for all $\epsilon>0$ 
there are $g,\pi\colon M\to M$, $g$ is a cw-expansive homeomorphism, $\pi$ is continuous, onto and satisfies
$\dist_{C^0}(f,g)<\epsilon$,
$\pi f=g\pi$ and
$\pi^{-1}(x)$ is a continuum of diameter less than $\epsilon$ for all $x\in M$.
\end{cor}

\begin{proof}
Recall that the residual set $G$ defined in \S\ref{secResidual} is the intersection of $F_n$ for $n\geq 1$. Also, $F_n=\cup_{0<\expc<1/n} E_\expc$ where 
\[
E_\expc=\left\{
f\in\homeo(M):
\sup_{i\in\Z}\diam f^i(C)\leq\expc\Rightarrow
\sup_{i\in\Z}\diam f^i(C)<\frac\expc 2,
\forall C\in \C(M)
\right\}.
\]
Thus, given $f\in G$ we have that $f$ is in $E_\expc$ for arbitrarily small values of $\expc$. In the terminology of \cite{AAV} this means that 
$f$ is half cw-expansive with expansivity constant $\expc$. 
Therefore, the result follows by \cite{AAV}*{Theorem 3.6}.
%
%
\end{proof}

In order to complement the picture of generic surface homeomorphisms let us mention some other known results. We recommend \cite{Pil} for more on this topic.
A generic surface homeomorphism satisfies
the pseudo-orbit tracing property and is not topologically stable \cite{Oda}.
Also, the non-wandering set has no explosions \cite{Ta1}, is the closure of the set of periodic points \cites{Cov,Pil} and is a Cantor set \cite{AHK}.
These are just a few known properties, more can be found in the references.

%

%
%
%
%


\end{document}